\documentclass[reqno, 11pt]{amsart}
\usepackage{color}
\usepackage{amsmath}
\usepackage{amsfonts}
\usepackage{amssymb}
\usepackage{amsthm}
\usepackage{newlfont}
\usepackage{graphicx}
\usepackage{algorithm,algorithmic,enumerate}

\usepackage{lineno,hyperref}

\usepackage[margin=1.35in]{geometry}

\newtheorem{thm}{Theorem}[section]
\newtheorem{prop}{Proposition}[section]
\newtheorem{lem}{Lemma}[section]
\theoremstyle{remark}
\newtheorem{rem}{Remark}[section]
\newtheorem*{rem*}{Remark}

\numberwithin{equation}{section}

\newcommand{\ed}{\end {document}}

\newcounter{smalllist}

\title[Quadratic convergence]{A global quadratic speed-up for computing the principal eigenvalue of Perron-like operators}
\author[D. Li]{Dong Li}
\address{D. Li, SUSTech International Center for Mathematics, and Department of Mathematics,  Southern University of Science and Technology, Shenzhen, P.R. China}%
\email{lid@sustech.edu.cn}
\author[J. Li]{Jianan Li}
\address{J. Li, Department of Mathematics,  Southern University of Science and Technology, Shenzhen, P.R. China}%
\email{11930493@mail.sustech.edu.cn}

\begin{document}

\begin{abstract}
We consider a new algorithm in light of the min-max Collatz-Wielandt formalism
 to  compute the principal eigenvalue and the eigenvector (eigen-function) for 
a class of positive Perron-Frobenius-like operators. Such operators are natural generalizations of the usual 
nonnegative primitive matrices. These have  nontrivial applications in PDE problems such as computing the principal eigenvalue of Dirichlet Laplacian
operators on general domains. We rigorously prove that for general initial data the corresponding
numerical iterates converge globally to the unique principal eigenvalue with quadratic convergence. We show that the
quadratic convergence is sharp with compatible upper and lower bounds.  We demonstrate the
effectiveness of the scheme via several illustrative numerical examples. 
\end{abstract}
\maketitle

\section{introduction}
A fundamental problem in the study of numerical linear algebra and operator theory 
is the computation of eigenvalue and associated eigenvectors/eigenfunctions. 
In the past century many powerful algorithms such as the power iteration, the Krylov
subspace methods, the Lanczos procedure, the Jacobi method, the LR and QR algorithms
(cf. \cite{G00} for a review)  have been developed to compute efficiently the eigenvalues and 
eigenvectors for given matrices.  The usual Perron-Frobenius theory is concerned with
a nonnegative primitive matrix which possesses a simple largest eigenvalue equal to its
spectral radius and all other eigenvalues have strictly smaller modulus.
 The associated eigenvector can be chosen to be completely positive, i.e.
all its entries are positive. For generic initial vector (i.e. carrying some nontrivial
component along the principal eigenvector direction), the standard power iteration method
 applied to a given nonnegative primitive matrix
leads to the convergence to the Perron eigen-pair with linear convergence speed. 
To be more concrete, consider $A \in \mathbb R^{m\times m}$ with nonnegative entries 
and assume $A$ is primitive. Denote by $\lambda_*=\rho(A)>0$ where $\rho(A)$ is the
spectral radius of $A$. Let $A \phi = \lambda_* \phi$ with $\phi>0$ (i.e. all entries of the
vector $\phi \in \mathbb R^m$ are positive), and $A^{\mathrm{T}} \psi
=\lambda_* \psi $ with $\psi>0$. The standard recipe for the power method is
\begin{enumerate}
\item Choose random $v^{(0)} \in \mathbb R^m$;
\item While $\| v^{(n+1) } -v^{(n)} \| >\epsilon_{\mathrm{tol}}$, calculate
\begin{align}
& y^{(n+1)} = A v^{(n)}; \quad v^{(n+1)} = \frac { y^{(n+1)} } {\| y^{(n+1) } \|}; \\
& \lambda^{(n+1)} =(v^{(n+1) } )^{\mathrm{T}} A v^{(n+1) }.
\end{align}
\end{enumerate}
Here $v^{(0)}$ is usually assumed to have some nontrivial projection along the principal eigenvector
$\phi$.  This assumption is almost innocuous in the measure sense.  On the other hand,
for $A$ being nonnegative and primitive, it suffices to take $v^{(0)}>0$ since $\phi$ is
completely positive. A slight variant of the power method is to take $\lambda > \lambda_*$ and
execute the following:

\underline{Algorithm $\lambda$-power ($\lambda$-inverse-power)}: 
\begin{enumerate}
\item Choose  $v^{(0)} >0$;
\item While $\| v^{(n+1) } -v^{(n)} \| >\epsilon_{\mathrm{tol}}$, calculate (below for
simplicity we write $\lambda I - A $ as $\lambda -A$)
\begin{align}
&  (\lambda-A) y^{(n+1)} =  v^{(n)}; \quad v^{(n+1)} = \frac { y^{(n+1)} } {\| y^{(n+1) } \|}; \\
& \lambda^{(n+1)} =(v^{(n+1) } )^{\mathrm{T}} A v^{(n+1) }.
\end{align}
\end{enumerate}
To ensure $\lambda>\lambda_*$ without the explicit value of $\lambda_*$, one can appeal to
the Collatz-Wielandt formalism, namely for any positive $y\in \mathbb R^n$, 
\begin{align} \label{1.5a}
\lambda_y^{-}:=\min_i  \frac { (Ay)_i} {y_i} \le \lambda_* \le \max_{i} \frac {(Ay)_i } {y_i}=:
\lambda_y^{+}.
\end{align}
More precisely, one has (below $y>0$ means $y_i>0$ for all $i$)
\begin{align}
\lambda_* = \inf_{y>0} \max_{i} \frac {(Ay)_i} {y_i} = \sup_{y>0} \min_{i} \frac {(Ay)_i}{y_i}.
\end{align}
Observe that by definition of $\lambda_y^+$ and $\lambda_y^-$, we have
\begin{align} \label{1.7a}
(A-\lambda_y^{-}) y \ge 0, \quad (\lambda_y^{+}-A) y \ge 0. 
\end{align}
The characterization \eqref{1.5a} follows easily
from taking inner product with  $\psi$ on both sides of \eqref{1.7a}.  Thanks to the above simple
characterization one can kick-start the algorithm by taking $\lambda > \lambda_{v^{(0)}}^+$. With
a fixed choice of $\lambda$ it is 
not difficult to check that for Algorithm $\lambda$-power  the speed of convergence toward the
principal eigenvector is linear.  

A natural next step is to consider somewhat more flexible choices of the parameter
$\lambda$. 
In light of Collatz-Wielandt we consider the  following iterative algorithm (cf. Chen \cite{Chen}).

\underline{Algorithm variable-$\lambda$-power}: 
\begin{enumerate}
\item Choose  $v^{(0)} >0$ and let $\lambda^{(0)} = 
\max_i \frac { (Av^{(0)})_i} { v^{(0)}_i} $. 
\item While $\| v^{(n+1) } -v^{(n)} \| >\epsilon_{\mathrm{tol}}$, calculate
\begin{align}
&  (\lambda^{(n)}-A) y^{(n+1)} =  v^{(n)}; \quad v^{(n+1)} = \frac { y^{(n+1)} } {\| y^{(n+1) } \|}; \label{1.8a}\\
& \lambda^{(n+1)} =\max_{i} \frac { (A v^{(n+1)} )_i} { v^{(n+1)}_i}. \label{1.9a}
\end{align}
\end{enumerate}

The purpose of this work, roughly speaking, is to give a rigorous analysis of the new algorithm
\eqref{1.8a}--\eqref{1.9a} in a general set-up.  As it turns out,
quite interestingly  the numerical iterates of the above algorithm converge globally (i.e. not requiring the initial data to be close to the principal eigen-pair) and super-linearly to the principal eigenvector. In fact the speed of convergence is quadratic with nearly optimal (up to some constant factor) upper and lower bounds. 
This upgrade of convergence speed also takes place
in the computation of eigen-pair for some infinite-dimensional operators. 
 We develop in detail the corresponding theoretical framework for a class of Perron-Frobenius-type operators which are
compact and positive. A prototypical example is $T= (-\Delta_D)^{-1}$ where $-\Delta_D$ is the usual
Dirichlet Laplacian on a smooth bounded domain in $\mathbb R^d$.  We employ the corresponding 
variable-$\lambda$-power algorithm for $T$ and calculate the principal eigenvalue and eigenfunction
of $-\Delta_D$. The convergence speed in this case is shown to be quadratic.

\begin{rem}
By now there are a plethora of  algorithms in the literature related to power and inverse power
methods. To put things into perspective, we mention several standard algorithms below the fold.
One should note that a common feature shared by these methods is that their convergence
is local, i.e. quadratic or even cubic convergence takes place only when the numerical iterates are close to some eigen-pair with sufficient precision. For simplicity assume $A\in \mathbb R^{m\times m}$ and has real eigenvalues $\lambda_*^1>\lambda_*^2\ge \cdots \ge \lambda_*^m$ with corresponding orthonormal eigenvectors $\phi_1$, $\cdots$, $\phi_m$.

\begin{itemize}
\item \underline{Algorithm inverse power with variable shift}: 
Choose $0\ne v^{(0)} \in \mathbb R^m$ and $\lambda^{(0)} \in \mathbb R$. Fix a given 
$z \in \mathbb R^m$. Recursively
execute the following:
\begin{align}
&\text{If $\lambda^{(n)} -A$ is singular then compute $(\lambda^{(n)} -A) v^{(n)} =0$ for some
nonzero $v^{(n)}$ and halt}; \\
& \text{Otherwise compute } y^{(n+1)} = (\lambda^{(n)} - A)^{-1} v^{(n)};   \; 
k^{(n+1)}= \frac 1 {z\cdot y^{(n+1)}} ; \\
& \quad \quad\quad \qquad \qquad\quad v^{(n+1)} = k^{(n+1)} y^{(n+1)} ; 
\quad \lambda^{(n+1)} = \lambda^{(n)}  - k^{(n+1)}. \label{1.12a}
\end{align}
In the above, a tacit working assumption is that $z \cdot y^{(n+1)} \ne 0$ for all iterates
$y^{(n+1)}$.  Define a function $F=F(v, \lambda):\; \mathbb R^m \times \mathbb R \to \mathbb
R^{m+1}$ as
\begin{align}
F(v, \lambda) = \begin{pmatrix}
(\lambda - A ) v \\
z^{\mathrm{T}} v - 1 
\end{pmatrix}.
\end{align}
The standard Newton iteration implemented on $F$ yields 
\begin{align} \label{1.14a}
\begin{pmatrix}
(\lambda^{(n)} -A) v^{(n)} \\
z^{\mathrm{T}} v^{(n)} -1
\end{pmatrix}
+ \begin{pmatrix}
\lambda^{(n)}-A \quad v^{(n)} \\
z^{\mathrm{T}} \quad 0 
\end{pmatrix}
\begin{pmatrix}
v^{(n+1)}-v^{(n)} \\
\lambda^{(n+1)} -\lambda^{(n)}
\end{pmatrix}
=\begin{pmatrix}
0 \\
0
\end{pmatrix}.
\end{align}
It is not difficult to verify that \eqref{1.14a} is equivalent to \eqref{1.12a}. As such the
algorithm \eqref{1.12a} exhibits local quadratic convergence to an eigen-pair.

\item \underline{Algorithm inverse power with Rayleigh quotient iteration}.
\begin{enumerate}
\item Choose  $v^{(0)} \in \mathbb R^m$ with $\| v^{(0)} \|_2=1$. Define 
$\lambda^{(0)}= (v^{(0)} )^{\mathrm{T}} A v^{(0)} $. 
\item  For $n\ge 0$, compute 
\begin{align}
&  (\lambda^{(n)}-A) y^{(n+1)} =  v^{(n)}; \quad v^{(n+1)} = \frac { y^{(n+1)} } {\| y^{(n+1) } \|}; \\
& \lambda^{(n+1)} =(v^{(n+1) } )^{\mathrm{T}} A v^{(n+1) }.
\end{align}
\end{enumerate}
It is well-known that (cf. \cite{T97}) if $A$ is a Hermitian matrix, then for almost all initial $v^{(0)}$,
the iterates converge to an eigen-pair. Furthermore the convergence is cubic.  One should note that
the key property used in the convergence proof is the minimal residual property, namely
$\lambda^{(n)}$ is the minimizer of the residual $\| (\lambda-A) v^{(n+1) } \|_2$ (for fixed $v^{(n+1)}$). In
\cite{P74}, Parlett extended the analysis to normal\footnote{The corresponding classification of
convergence is more complicated. See for example Theorem on pp.685 of \cite{P74}.}
 and nonnormal matrices.
\end{itemize}

\end{rem}

It is worthwhile mentioning some other approaches for computing eigen-pairs of the aforementioned Perron-Frobenius-like operators. In \cite{LM07} Lejay and Maire considered a Monte Carlo method
for the numerical computation of the principal eigenvalue of the Dirichlet Laplacian in a bounded
piecewise smooth domain.  The main idea is based on the characterization
\begin{align}
\lambda_1 = \lim_{t\to +\infty} \frac 1 t \log \mathbb P ( \tau_D^x >t),
\end{align}
where $\lambda_1$ is the principal eigenvalue for $\Delta_D$, and $\tau_D^x$ is the exit time
from $D$ of the Brownian motion starting at $x$.  Various numerical schemes such as the Euler
scheme, walk on sphere schemes and walk on rectangles schemes are considered in \cite{LM07}. 
In \cite{LW20}, Li and Wang considered a family of Perron-like matrices and investigated the principal
eigenvalues and the associated generalized eigen-spaces via polynomial approximations of
matrix exponentials.  For some closely-related developments on the computation of eigen-pairs for
matrices, we refer to \cite{At,BaiZJ,Bra,Buns,FH,Ford,Wat1,Wat2,Wat3,Wen}
and the references therein for more extensive discussions.

The rest of this paper is organized as follows. In Section 2 we develop the general theory for
the variable $\lambda$-power algorithm applied to a class of Perron-Frobenius-like operators. 
In Section 3 we carry out several numerical experiments to showcase the effectiveness
of the new algorithm. The examples range from the usual primitive nonnegative matrix to
Dirichlet Laplacian operators on polygonal-type domains.

\section{computation of the principal eigenvalue: Dirichlet Laplacian case}

We first introduce a general set-up. 
Let $\Omega$ be a bounded domain in $\mathbb R^d$ with smooth
boundary. Denote $\mathbb H= L^2(\Omega)$ which consists real-valued $L^2$-integrable
functions on $\Omega$.  We denote by $\langle \cdot, \cdot \rangle $ the usual
$L^2$ inner product. Let  $T= (-\Delta_D)^{-1}$ where $\Delta_D$ is the  Dirichlet Laplacian.

 We note that $T$ satisfies the following conditions:
\begin{enumerate}
\item $T \phi= \lambda^* \phi $, where $\lambda^*>0$ is the largest eigen-value of $T$ and $\phi$ is
the corresponding eigen-vector.  By the strong maximum principle, $\phi(x)$ must take a constant sign in $\Omega$ and with no loss we may assume
 $\phi(x)>0$ for all $x \in \Omega$. We may also assume $\| \phi \|_2 =1$. 

\item $T$ is a completely positive operator, namely if $v\ge 0$ and $v$ is not
identically zero, then  $(T v )(x)> 0$ for all $x \in \Omega$. 
\end{enumerate}
\begin{rem}
One can consider more generally transfer operators, namely Ruelle--Perron--Frobenius transfer
operators. One can also generalize to the usual Krein-Rutman setup but we shall not pursue this more
abstract situation here.
\end{rem}
\begin{rem}
Note that $1/\lambda^*$ is the usual principal eigen-value of $-\Delta_D$ which is a simple eigenvalue.
\end{rem}

\begin{lem}[Monotonicity of the one-step growth-factor] \label{lem2.1}
Let $B=(\lambda -T)^{-1}$ where $\lambda>\lambda^*$.  Suppose $v\ge 0$, $v$ is not identically
 zero and let 
$w= B v$. Then $w(x)>0$ for all $x \in \Omega$, and 
\begin{align}
\lambda^*\le \sup_{x\in \Omega} \frac { (Tw)(x) } {w(x)} \le \sup_{x:\, v(x) >0} \frac {(Tv)(x)} {v(x)}.
\end{align}
\begin{rem}
Roughly speaking, this lemma asserts that averaging decreases the one-step growth-factor. 
\end{rem}
\end{lem}
\begin{proof}
Note that the operator $B$ commutes with $T$.  Set $l=\sup_{x:\, v(x) >0} \frac {(Tv)(x)} {v(x)}$
and $l_w= \sup_{x\in \Omega} \frac {(Tw)(x)} {w(x)}$.
With no loss we can assume $l <\infty$. By definition we have $(l -T) v \ge 0$.  On the other hand
\begin{align}
(l-T) w = (l-T) (\lambda-T)^{-1} v = (\lambda- T)^{-1} (l-T) v \ge 0.
\end{align}
Thus $l\ge l_w$. To see $l_w \ge \lambda^*$ we argue as follows. 
Clearly by definition $(l_w -T) w \ge 0$. Taking the inner produce with $\phi$, we obtain
(below recall $\langle\cdot , \cdot\rangle$ is the usual $L^2$ inner product)
\begin{align}
(l_w- \lambda^*) \langle w, \phi \rangle \ge 0.
\end{align}
Since $\langle w, \phi \rangle >0$,  we obtain $l_w\ge \lambda^*$.
\end{proof}

We consider the following algorithm:

Take nontrivial  smooth nontrivial $v^0$ such that $v^0(x)>0$ for all $x\in \Omega$.
Assume that
\begin{align}
\lambda^{(0)}:=\sup_{x\in \Omega } \frac{(Tv^0)(x)} {v^0(x)} <\infty.
\end{align}

Iteratively define 
\begin{align}
&v^{n+1} = (\lambda^{(n)} - T)^{-1} (\lambda^{(n)} - \lambda_*)v^n; \notag\\
&\lambda^{(n+1)} = \sup_{x:\, v^{n+1}(x)>0} \frac{(Tv^{n+1})(x) } {v^{n+1}(x) }.
\end{align}
It should be noted that it is not difficult to check that $v^{n+1}(x)>0$ for all
$n\ge 0$, and $x \in \Omega$. Thus in the above definition of $\lambda^{(n+1)}$
the constraint $\{x:\, v^{n+1}(x)>0\}$ can be replaced by
the innocuous condition $\{x:\, x \in \Omega\}$.  By Lemma \ref{lem2.1}
it is also clear that 
\begin{align}
\lambda^*\le \lambda^{(n+1)} \le \lambda^{(n)} \le \lambda^{(0)}<\infty
\end{align}
for all $n\ge 0$.  A somewhat delicate technical subtlety is that we need to enforce $\lambda^{(n)}>\lambda^*$ for all $n$
 so that all the iterates remain well-defined. The next lemma clarifies this issue.
 
 \begin{lem}
 The following hold.
 \begin{enumerate}
 \item If $\inf_{x\in \Omega} \frac {(Tv^0)(x)} {v^0(x)} = \sup_{x\in \Omega} \frac {(Tv^0)(x)} {v^0(x)}$, then
 $v^0(x)= c_1 \phi(x)$, where $c_1>0$ is some constant.
 \item If $\inf_{x\in\Omega} \frac {(Tv^0)(x)} {v^0(x)} <\sup_{x\in\Omega} \frac {(Tv^0)(x)} {v^0(x)}$,
 then $\lambda^{(0)}>\lambda^*$.
 \item  If $\inf_{x\in\Omega} \frac {(Tv^0)(x)} {v^0(x)} <\sup_{x\in\Omega} \frac {(Tv^0)(x)} {v^0(x)}$, then for all $n\ge 0$, we have
 $\inf_{x\in\Omega} \frac {(Tv^n)(x)} {v^n(x)} <\sup_{x\in\Omega} \frac {(Tv^n)(x)} {v^n(x)}$
 and
  the strict inequalities:
 \begin{align}
 \lambda^*<\lambda^{(n+1)} <\lambda^{(n)} < \lambda^{(0)}.
 \end{align}
 \end{enumerate}
 \end{lem}
\begin{proof}
(1) Clearly in this case we have $Tv^0 = \lambda^{(0)} v^0$.  Taking the inner product with $\phi$ yields
that $\lambda^{(0)}=\lambda^*$.  Since $\lambda^*$ is a simple eigen-value, we must have $v^0 =c_1 \phi$ for some $c_1>0$.

(2)  We employ an orthonormal  eigen-basis $(\phi_j)_{j=1}^{\infty}$ for $T$. Namely
  $\phi_1=\phi$, $\lambda_1^*=\lambda^*$, and $\|\phi_j\|_2=1$ satisfies
$T\phi_j = \lambda_j ^*\phi_j$, $\lambda_{j}^*\le \lambda_{j-1}^*$ for all $j\ge 2$. Furthermore
there is the strict spectral gap $\lambda_2^*<\lambda_1^*=\lambda^*$. Clearly by assumption we have
\begin{align}
v^0(x) = c_1 \phi (x) + \sum_{j=2}^{\infty} c_j \phi_j (x),
\end{align}
where $c_1=\langle v^0, \phi \rangle >0$, and $ \sum_{j=2}^{\infty} c_j^2 >0$.  We then write
\begin{align}
(Tv^0)(x) - \lambda^* v^0 (x) = \sum_{j=2}^{\infty} c_j (\lambda_j^*-\lambda^*) \phi_j(x)=:
\beta(x).
\end{align}
Clearly
\begin{align}
\| \beta\|_2^2 = \sum_{j=2}^{\infty} c_j^2 (\lambda_j^*-\lambda^*)^2 >0
\end{align}
so that $\beta$ is not identically zero.  On the other hand $\langle \beta, \phi \rangle =0$ which
implies that $\beta$ must change sign on $\Omega$. In particular we must have
\begin{align}
\sup_{x \in \Omega} \frac {\beta(x) }{v^0(x)} >0.
\end{align}
This implies that $\lambda^{(0)}>\lambda^*$.

(3) Since $v^0$ cannot be a multiple of $\phi$, we obtain $v^1$ cannot be a multiple of $\phi$.
By a similar analysis as in Step 2, we obtain $\lambda^{(1)}>\lambda^*$. A simple induction
argument yields $\lambda^{(n)} >\lambda^*$ for all $n$.  Next we show the strict inequality
$\lambda^{(n)} >\lambda^{(n+1)}$. By definition we have
\begin{align}
\lambda^{(n)} -\lambda^{(n+1)}
& = \lambda^{(n)} - \sup_{x\in \Omega} \frac {Tv^{n+1} } {v^{n+1} } \\
& =  (\lambda^{(n)} - \lambda^*) \inf_{x\in \Omega}  \frac {  v^n (x) } {v^{n+1}(x)}
=\inf_{x\in \Omega} \frac {v^n (x) } { ( (\lambda^{(n)} -T)^{-1} v^n ) (x)}.  \label{2.13a}
\end{align}
It suffices for us to show $\sup_{x\in \Omega} \frac { ((\lambda^{(n)}-T)^{-1} v^n )(x)}{v^n(x)} <\infty$.
By an argument similar to the proof of Lemma \ref{lem2.1}, we reduce the matter to showing
\begin{align} \label{2.14a}
\sup_{x\in \Omega} \frac {( (\lambda^{(n)} - T)^{-1} v^0 (x) } { v^0 (x) } <\infty.
\end{align}
Observe that
\begin{align}
&( (\lambda^{(n)}-T)^{-1} v^0) (x) - (\lambda^{(n)})^{-1} v^0 (x)
= (T \underbrace{\int_0^1 (\lambda^{(n)} -\theta T)^{-2} v^0 d\theta}_{=:\tilde u }) (x) ;\\
& \sup_{x\in \Omega} \frac {(T\tilde u)(x) } {v^0(x)} \le
\sup_{x\in \Omega} \frac { (Tv^0)(x) } {v^0(x)} \cdot \sup_{x\in \Omega}
\frac { (T\tilde u)(x) } {(Tv^0)(x) }.
\end{align}
By strong maximum principle, it is not difficult to check that
\begin{align}
&\inf_{x\in \Omega} \frac {(Tv^0)(x) } {\mathrm{dist}(x, \partial \Omega)} \ge 
\alpha_1>0, \quad \sup_{x \in \Omega} \frac{(T\tilde u)(x) } {\mathrm
{dist} (x, \partial\Omega) } \le \alpha_2, \\
&\sup_{x\in\Omega} \frac {(T\tilde u)(x) } {(Tv^0)(x) } \le \frac {\alpha_2}{\alpha_2} <\infty,
\end{align}
where in the above $\alpha_1>0$, $\alpha_2>0$ are constants. 
Thus $\inf_{x\in \Omega} \frac {v^n (x) } { ( (\lambda^{(n)} -T)^{-1} v^n ) (x)}>0$.
\end{proof}

From the computational point of view, the definition of $\lambda^{(n+1)}$ is somewhat unwieldy
since it involves a further computation of $T v^{n+1}$ which could be costly. 
A close inspection of \eqref{2.13a} suggests that the following variant is slightly better:
\begin{align}
& \mu^{(0)} = \lambda^{(0)}, \quad \tilde v^0 = v^0; \\
& \tilde v^{n+1} = (\mu^{(n)} - T)^{-1} \tilde v^n; \\
& \mu^{(n+1)} = \mu^{(n)} - \inf_{x\in \Omega} \frac {\tilde v^n(x) } {\tilde v^{n+1}(x) }.
\end{align}
In the practical numerical computation, to ensure numerical stability 
we often normalize $\tilde v^n$ by its $L^2$ norm (more precisely,  the  discrete $l^2$ norm after full numerical discretization) at each
iteration step. Clearly this would not change the definition of $\mu^{(n+1)}$.  A very pleasing 
feature of the new definition  is that the inf-part  no longer involves any further 
computation of the $T$ operator.   Not surprisingly we have the following.

\begin{prop}[Equivalence of $\lambda^{(n)}$ and $\mu^{(n)}$]
Assume $\mu^{(0)}=\lambda^{(0)}$, $\tilde v^0 = v^0$ and $$\inf_{x\in\Omega} \frac {(Tv^0)(x)} {v^0(x)} <\sup_{x\in\Omega} \frac {(Tv^0)(x)} {v^0(x)}<\infty.$$
The following hold.
\begin{enumerate}

\item $\mu^{(n)} = \lambda^{(n)}$ for all $n\ge 1$. Furthermore $\tilde v^n =c_n v^n$ 
for all $n\ge 1$, where $c_n>0$ is a constant depending on $n$. 

\item For all $n\ge 0$, we have
\begin{align}
0< \inf_{x\in \Omega} \frac {\tilde v^n (x) } {\tilde v^{n+1}(x)} <\mu^{(n)} -\lambda_*.
\end{align}
\end{enumerate}
\end{prop}
\begin{proof}
The first statement follows from an easy induction argument. The second statement follows
from the properties of $\lambda_n$ established earlier.
\end{proof}

\begin{thm}[Quadratic convergence of $\lambda^{(n)}$]
Suppose $$\inf_{x\in\Omega} \frac {(Tv^0)(x)} {v^0(x)} <\sup_{x\in\Omega} \frac {(Tv^0)(x)} {v^0(x)}=\lambda^0<\infty.$$

The following hold.
\begin{enumerate}
\item $\lambda^* <\lambda^{(n+1)} <\lambda^{(n)} <\lambda^0$ for all $n$, and
$\lambda^{(n)} \to \lambda^*$ as $n\to \infty$.
\item There exist some constants $C>c>0$, and integer $n_0\ge 2$, such that for all
$n\ge n_0$, we have 
\begin{align}
0< c\le \frac{ \lambda^{(n+1)} - \lambda^*} { (\lambda^{(n)} -\lambda^*)^2} \le C <\infty.
\end{align}
\end{enumerate}
\end{thm}
\begin{proof}
(1) We show $\lambda^{(n)} \to \lambda^*$ as $n\to \infty$. We argue by contradiction. Suppose
$\lambda^{(n)} \to \mu>\lambda^*$.  By \eqref{2.13a}, we obtain
\begin{align}
\lambda^{(n)} -\lambda^{(n+1)} \ge \inf_{x\in \Omega} \frac {v^n (x) } { ( (\lambda^{(n)} -T)^{-1} v^n ) (x)}.
\end{align}
By an argument similar to the proof of Lemma \ref{lem2.1}, we have
\begin{align}
\sup_{x \in \Omega} 
\frac { ((\lambda^{(n)} - T)^{-1} v^n )(x) } {v^n(x)}
\le \sup_{x \in \Omega}
\frac { ((\lambda^{(n)} - T)^{-1} v^0 )(x) } {v^0(x)}.
\end{align}
Since $\lambda^{(n)} \ge \mu >\lambda^*$ and
\begin{align}
(\lambda^{(n)} - T)^{-1} - (\mu - T)^{-1} = -(\lambda^{(n)} -\mu) (\lambda^{(n)}-T)^{-1}
(\mu - T)^{-1},
\end{align}
we clearly have
\begin{align}
\sup_{x \in \Omega}
\frac { ((\lambda^{(n)} - T)^{-1} v^0 )(x) } {v^0(x)} 
\le \sup_{x \in \Omega}
\frac { ((\mu - T)^{-1} v^0 )(x) } {v^0(x)} \le \beta<\infty,
\end{align}
where $\beta>0$ is independent of $n$, and the last inequality follows a similar proof
of \eqref{2.14a}. It follows that $\lambda^n - \lambda^{n+1} \ge c_{\beta}>0$ for some
fixed constant $c_{\beta}>0$.  This clearly leads to a contradiction.

(2) 
Since $v^0$ is not a multiple of $\phi$, we may (after a simple normalization) write
\begin{align}
v^0 = \phi + \sum_{j=j_0}^{\infty} c_j^{(0)} \phi_j,
\end{align}
where $T \phi_j = \lambda_j^* \phi_j$, and $c_{j_0}^{(0)} \ne 0$ ($j_0\ge 2$) is the first nonzero
coefficient.  Here if $\lambda_j = \lambda_{j_0}$ for $j_0 \le j \le j_1$ (i.e. taking into account
the multiplicity)  we tacitly assume $c_{j_0}^{(0)}$ is the largest amongst $|c_{j}^{(0)}|$ for
$j_0\le j\le j_1$. With no loss we assume $c_{j_0}^{(0)} >0$.

The eigen-functions $\phi_j$ satisfy $T \phi_j = \lambda_j \phi_j$
with $\lambda_{j+1} \le \lambda_j$ for all $j\ge 2$. By a simple induction, we have
\begin{align}
v^{n} = \phi + \sum_{j=j_0}^{\infty} c_j^{(n)} \phi_j, 
\qquad c_j^{(n)} = c_j^{(0)} \prod_{l=0}^{n-1} \frac {\lambda^{(l)} -\lambda^*}{\lambda^{(l)}
-\lambda_j^*}.
\end{align}
By using the definition of $\lambda^{(n+1)}$, we obtain
\begin{align}
\frac {\lambda^{(n+1)} - \lambda^*}
{\lambda^{(n)} - \lambda^*}
= \sup_{x\in \Omega}
\frac { ((T-\lambda^*) v^{n+1})(x)} {v^{n+1}(x)}
=\sup_{x\in \Omega}
\frac{ \sum_{j=j_0}^{\infty} c_j^{(n+1)} (\lambda_j^*-\lambda^*) \phi_j }
{ \phi + \sum_{j=j_0}^{\infty} c_j^{(n+1) } \phi_j }.
\end{align}
It is not difficult to check that for $n$ sufficiently large, 
\begin{align}
&K_1\le \frac {\phi (x)+ \sum_{j=j_0}^{\infty} c_j^{(n+1)} \phi_j (x) } {\phi  (x)}
\le K_2, \quad \forall\, x \in \Omega;\\
&\| \frac{\sum_{j=j_0}^{\infty} c_j^{(n+1) } (\lambda_j^*-\lambda^*) \phi_j } {\phi } \|_{L^{\infty}(\Omega)}
\le c_{j_0}^{(n+1)} K_3,
\end{align}
where $K_i>0$, $i=1,2,3$ are constants independent of $n$.   These imply that
\begin{align}
\frac {\lambda^{(n+1)} - \lambda^*}
{\lambda^{(n)} - \lambda^*}
\le \; \mathrm{const} \cdot c_{j_0}^{(n+1)} .
\end{align}
On the other hand, we write
\begin{align}
  & \sum_{j=j_0}^{\infty} c_j^{(n+1)} (\lambda_j^*- \lambda^*) \phi_j \notag \\
=& (\lambda_{j_0}^* - \lambda^*) (\prod_{l=0}^n \frac{ \lambda^{(l)} -\lambda^*}
{\lambda^{(l)} - \lambda_{j_0}^*} )
\sum_{j=j_0}^{j_1}  c_j^{(0)} \phi_j  + \sum_{j>j_1} c_j^{(n+1)} (\lambda_j^*-\lambda^*) \phi_j.
\end{align}
Here we  $j_1\ge j_0$ accounts for the multiplicity of the eigen-value $\lambda^*_{j_0}$,
namely we have $\lambda_j^*= \lambda^*_{j_0}$ for $j_0\le j\le j_1$, and
$\lambda_j^*<\lambda_{j_0}^*$ for $j>j_1$.  Thanks to the spectral gap, the coefficients 
$c_j^{(n+1)}$ decays faster than $c_{j_0}^{(n+1)}$ for $j>j_1$.  Clearly we can fix some
$x_0 \in \Omega$ such that
\begin{align}
d_{x_0}:= \sum_{j=j_0}^{j_1} c_j^{(0)} \phi_j (x_0) >0.
\end{align}
This is indeed possible since 
\begin{align}
\| \sum_{j=j_0}^{j_1} c_j^{(0} \phi_j \|_{L^2(\Omega)}^2 
\ge (c_{j_0}^{(0)} )^2 >0, \quad \langle \sum_{j=j_0}^{j_1}
c_j^{(0)} \phi_j, \phi \rangle =0.
\end{align}
It follows that for $n$ sufficiently large
\begin{align}
\sup_{x\in \Omega}
\frac { \sum_{j=j_0}^{\infty} c_j^{(n+1)} (\lambda_j^*-\lambda^*) \phi_j (x)}
{ \phi (x)} \ge K_4 c_{j_0}^{(n+1)},
\end{align}
where $K_4>0$ is a constant independent of $n$.  Thus we have shown for 
$n$ sufficiently large:
\begin{align}
\frac{ \lambda^{(n+1)} - \lambda^*}{ \lambda^{(n)} -\lambda^*} 
\sim c_{j_0}^{(n+1)} \sim \prod_{l=0}^n \frac { \lambda^{(l)} - \lambda^*} { \lambda^{(l)}
-\lambda_{j_0}^*}.
\end{align}
The desired result follows easily.
\end{proof}

\section{Numerical experiments}
\subsection{Matrix case}

In this subsection, we first apply the aforementioned algorithm to compute the principal eigenvalue of primitive matrices. Let $A$ be a nonnegative primitive matrix. We shall implement the
 algorithm \ref{alg nonnegative} specified below.
\begin{algorithm}[h]
	\caption{}
	\label{alg nonnegative}
	\begin{algorithmic}
		
		\STATE {
			\begin{enumerate}[\textbf{Step} 1]
				\item
				Choose the initial vector $v^{0}>0$ and compute
				\begin{equation}\label{alg1 z0}
				\lambda^{(0)}= \max_{1\leqslant i \leqslant N} \frac{[A v^{0}]_i}{v^{0}_i}
			    \end{equation}
				
				\item
				{For $n=1,2,\dots$}, solve the linear equation
				\begin{equation}\label{alg1 system}
				\Big(\lambda^{(n-1)}I-A \Big)\omega^{n}=v^{n-1}
			    \end{equation}
				and calculate (below $\|\cdot \|$ denotes the usual $l^2$ norm) 
				\begin{align}
				v^{n}=\frac {\omega^{n}}{\|\omega^{n}\|}.
				\end{align}
				
				\
				Update $\lambda^{(n)}$ with
				\begin{equation}\label{alg1 zn}
				\lambda^{(n)}= \max_{1\leqslant i \leqslant N} \frac{[A v^{n}]_i}{v^{n}_i}.
			    \end{equation}
				
				\item
				Enforce the stopping criterion and output 
				$ (\lambda^{(n)}, \; v^{n})$.
		\end{enumerate}}
	\end{algorithmic}
\end{algorithm}

We take $A=H$ or $B$, where $H$ is the well-known Hilbert matrix with entries $H_{ij}=(i+j-1)^{-1}$, $1\le i,j\le N$,  and $B$ is a nonnegative random symmetric tridiagonal matrix
\begin{equation}
	B= \begin{pmatrix}
	a_1 &   b_1&   0& 0&    \cdots\\
	b_1&a_2&b_2&0 &\cdots \\
	0&b_2&a_3&b_3&\cdots \\
	\vdots &\vdots & \vdots&\ddots&\vdots\\
	0&0&0&b_{N-1}&a_N\\
\end{pmatrix}.
\end{equation}
In the above $a_i\sim U(0,2)$ are (iid) uniformly randomly distributed  on the interval $(0,\,2)$, and $b_i\sim U(0,1)$ are (iid) uniformly randomly distributed on the interval $(0,\,1)$ . We take $N=1000$. We shall compute the principal eigenvalue of the matrix $A$. 
\begin{rem}
The Hilbert matrix naturally arises from the least square approximation of general $L^2$ functions on $(0,1)$ via the polynomial basis $\{ 1, x, \cdots, x^{N-1} \}$, namely
\begin{align} 
\mathcal E =\min_{a_0, \cdots, a_{N-1} } \int_0^1 (f(x) -q_N(x) )^2 dx,
\end{align}
where $q_N(x) =\sum_{j=0}^{N-1} c_j x^j$. The critical point equations  yield
\begin{align}
\sum_{j=1}^N c_{j-1} \int_0^1 x^{i+j-2} dx  =\underbrace{\int_0^1 x^{i-1} f(x) dx}_{=:f_i}.
\end{align}
This in turn leads to the linear system of equations $H c= f$.  A well-known issue is 
that $H$ is ill-conditioned.

\end{rem}

\textbf{Stopping criterion.} 
In light of the Collatz-Wielandt bound, we can take 
\begin{equation}
	\text{SC1}=\max_{i} \frac{A v^{n}}{v^{n}} - \min_{i} \frac{A v^{n}}{v^{n}}< \varepsilon
	\label{mat crit 1}
\end{equation}
as a stopping criterion, where $\varepsilon$ is the tolerance threshold. However, for tridiagonal matrices, the principal eigenvector sometimes admit extremely small components and the 
LHS of  \eqref{mat crit 1} can remain $O(1)$ when the numerical iterates are in the vicinity of the desired eigen-pair (See Table \ref{tab matrix}). In lieu of SC1, we enforce the following  stopping criterion
\begin{equation}
	\label{3.7a}
	|\lambda^{(n+1)}-\lambda^{(n)}|<\varepsilon.
\end{equation}
%

\begin{rem}
	We explain why the LHS of \eqref{mat crit 1} can stay $O(1)$ whilst $(\lambda^{(n)},\,v^n)$ is close to $(\lambda_*,\,v)$, where $(\lambda_*,\,v)$ is the principal eigen-pair. Roughly speaking
	\begin{equation}
		r^{n}_i=\frac{(Av^{n})_i}{v^{n}_i}=\sum_{j} a_{ij}\frac{v^{n}_j}{v^{n}_i}
	\end{equation}
which clearly depends on the ratios $\frac{v^{n}_j}{v^{n}_i}$. In the scenario where $v=(v_i)$ has vastly different components, the ratio $\frac{v^{n}_j}{v^{n}_i}$ can become very large and $\left|\frac{v^{n}_j}{v^{n}_i}-\frac{v_j}{v_i}\right|\gtrsim 1$ whilst $v^n$ is close to $v$. As a result the round-off errors accumulated in intermediate computations could spoil the convergence of SC1. In general, it will take more iterations for SC1 to converge than \eqref{3.7a}.
\end{rem}

\textbf{Number of iterations.} We take the initial vector $v^0=(1,\,\cdots,\,1)^T$. The exact principal eigenvalue is computed by the function \texttt{eig()} of MATLAB. In general it takes $6$ to  $8$ iterations for our
algorithm to reach the stopping criterion (See Table \ref{tab matrix}).

\textbf{Quadratic convergence. }
As we know the definition of quadratic convergence is
\begin{equation}
	\frac{|\lambda^{(n+1)}-\lambda|}{|\lambda^{(n)}-\lambda|^2} \to c>0
\end{equation}
as $n\rightarrow \infty$. Thus \footnote{Denote $a_{n+1}=|\lambda^{n+1}-\lambda|$. Clearly one expects $a_{n+1}=(c+\varepsilon_n)a_{n}^2$ $\Rightarrow$ $\log(a_{n+1})=2\log(a_n)+\log(c+\varepsilon_n)$ $\Rightarrow$ $\frac{\log(a_{n+1})-\log(a_n)}{\log(a_n)-\log(a_{n-1})}\rightarrow 2$}
\begin{equation}
	\frac{\log(|\lambda^{(n+1)}-\lambda|)-\log(|\lambda^{(n)}-\lambda|)}{\log(|\lambda^{(n)}-\lambda|)-\log(|\lambda^{(n-1)}-\lambda|)} \rightarrow 2
	\label{test matrix order}
\end{equation}
as $n\rightarrow \infty$. As such we shall appeal to the above formula to test the convergence rate (See Table \ref{tab matrix}). Reassuringly in Table \ref{tab matrix} the  quadratic convergence 
takes place when $v^{n}$ is sufficiently close to the exact eigenvector. 
\begin{rem}
One should note that in Table \ref{tab matrix}, the row corresponding to the iteration step 8  belongs
to the realm of the machine error.  As such the order computed in this row which carries the numerical
value  $1.193$ should be treated as an outlier. 
\end{rem}
\textbf{Comparison with Rayleigh quotient.}
For comparison 
we apply the usual Rayleigh quotient iteration to aforementioned two matrices. We take $v^0=$\texttt{ones(N,1), rand(N,1)},
with $N=100$ or $N=1000$. In yet other words $v^0$ consists of identical ones, or iid random variables on $(0,1)$. Quite interestingly, it is observed that none of them converge to the principal eigenvalue. This corroborates very well with the folklore fact that the Rayleigh quotient iteration usually converges locally.

On the other hand, if we consider a small Hilbert matrix $H$ with $N=50$ and $v^0=1$, the cubic convergence to the principal eigen-pair takes place within $6$ steps. 
See the last two columns of Table \ref{tab matrix}.  We shall use \eqref{3.7a} as the stopping 
criterion. 

\begin{algorithm}[h]
	\caption{Rayleigh quotient iteration}
	\label{alg rayl nonnegative}
	\begin{algorithmic}
		
		\STATE {
			\begin{enumerate}[\textbf{Step} 1]
				\item
				Choose the initial vector $v^{0}>0$ and compute
				\begin{equation}\label{alg rl z0}
					\lambda^{(0)}= \frac{(v^0)^TAv^0}{(v^0)^Tv^0}
				\end{equation}
				
				\item
				{For $n=1,2,\dots$}, solve the linear equation
				\begin{equation}\label{alg rl system}
					\Big(\lambda^{(n-1)}I-A \Big)\omega^{n}=v^{n-1}
				\end{equation}
				and calcuate 
				\begin{align}
					v^{n}=\frac {\omega^{n}}{\|\omega^{n}\|}.
				\end{align}
				
				\
				Update $\lambda^{(n)}$ with
				\begin{equation}\label{alg rl zn}
					\lambda^{(n)}= (v^n)^TAv^n.
				\end{equation}
				
				\item
				Enforce the stopping criterion and output 
				$ (\lambda^{(n)}, \; v^{n})$.
		\end{enumerate}}
	\end{algorithmic}
\end{algorithm}

\begin{table}[htb]
	\centering  %
	\begin{tabular}{|c|c|c|c|c|c|c|c|c|}  
		\hline
		& \multicolumn{3}{|c|}{Tridiagonal} &\multicolumn{3}{|c|}{Hilbert matrix}&\multicolumn{2}{|c|}{Hilbert matrix, Rayleigh}\\ \hline
		Iterations&$\lambda^{(n)}-\lambda$&Order&SC1&$\lambda^{(n)}-\lambda$&Order&SC1&$|\lambda^{(n)}-\lambda|$&Order\\ \hline
		1&4.292e-02& --  &2.994&0.993    &--   &4.077&0.2915&--\\
		2&4.173e-03& --  &2.771&0.441    &--   &2.517&0.1941&--\\
		3&5.768e-05&1.837&2.658&0.160    &1.246&1.382&4.0983e-02&3.830\\
		4&1.571e-08&1.917&2.522&3.627e-02&1.473&0.409&1.8191e-04&3.482\\
		5&8.449e-16&2.039&2.322&2.611e-03&1.766&2.638e-02&1.3255e-11&3.033\\
		6&1.408e-16&0.107&2.228&1.482e-05&1.965&1.213e-04&8.5554e-16&0.587\\
		7&--       & --  & --  &4.689e-10&2.003&3.347e-09&--&--\\
		8&--       & --  & --  &1.999e-15&1.193&1.110e-14&--&--\\
		\hline			
	\end{tabular}
	\caption{The convergence error, order, stopping criterion  
	 v.s. iteration step.}
	\label{tab matrix}
\end{table}

\subsection{The Dirichlet Laplacian case}

In this subsection, we apply aforementioned algorithm to compute the principal eigenvalue  of $T=(-\Delta_D)^{-1}$, where $-\Delta_D$ is Dirichlet Laplacian operator. We shall implement the
algorithm \ref{alg pde nonnegative} specified below. To expedite the computation, we shall employ
the following identity without explicit mentioning:
\begin{align}
(\lambda-T)^{-1}=(-\Delta_D)(\lambda (-\Delta_D)-1)^{-1}.
\end{align}
 In yet other words, we shall employ $(-\Delta_D)(\lambda (-\Delta_D)-1)^{-1}$ to compute $w^{n+1}$.

\begin{algorithm}
	\caption{}
	\label{alg pde nonnegative}
	\begin{algorithmic}
		
		\STATE {
			\begin{enumerate}[\textbf{Step} 1]
				\item
				Choose the initial vector $v^{0}>0$ and compute
				\begin{equation}
				\label{alg1 pde z0}
				\lambda^{(0)}= \sup_{x \in \Omega} \frac{Tv^{0}}{v^{0}}
				\end{equation}
				
				\item
				{For $n=0,\,1,\,2,\,\dots$}, solve the linear equation
				\begin{equation}
				\label{alg1 pde system}
				\omega^{n+1}=\Big(\lambda^{(n)}I-T \Big)^{-1}v^{n}
				\end{equation}
				and calcuate 
				\begin{align}
					v^{n}=\frac {\omega^{n}}{\|\omega^{n}\|}.
				\end{align}

				Update $\lambda^{(n+1)}$ with
				\begin{equation}
				\label{alg1 pde zn}
				\lambda^{(n+1)}= \sup_{x \in \Omega} \frac{Tv^{n+1}}{v^{n+1}},
			    \end{equation}
		        or
		        \begin{equation}
		        	\lambda^{(n+1)}=\lambda^{(n)}-\inf_{x\in \Omega} \frac{v^n}{w^{n+1}}
		        \end{equation}
				
				\item
				Enforce the stopping criterion and output 
				$ (\lambda^{(n)}, \; v^{n})$.
		\end{enumerate}}
	\end{algorithmic}
\end{algorithm}

\textbf{Domain.} We take the two dimensional domain $\Omega$ as specified in Figure \ref{domain}. This very nice example is take from \cite{LM07} (see Section 4.1 therein). 
\begin{figure}[htbp]
	\begin{center}
		\includegraphics[width=5cm]{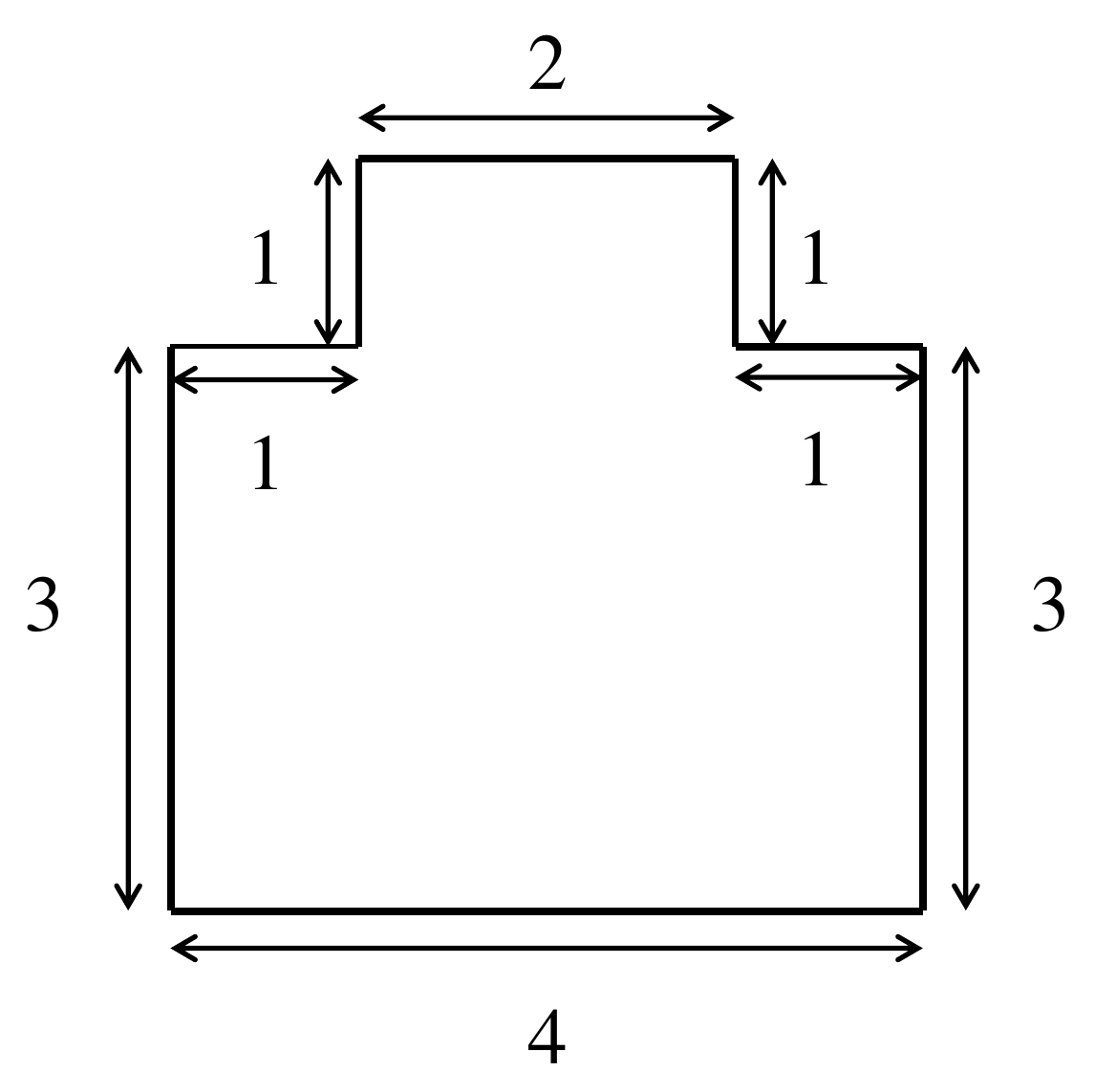}
	\end{center}
	\vspace{-2em}
	\caption{ The domain $\Omega$.}
	\label{domain}
\end{figure}

\textbf{Discretiztion.} We use the finite difference method to compute the Dirichlet Laplacian operator. Denote $T_h=(-\Delta)_h^{-1}$, where $\Delta_h$ is the discrete Laplacian operator with $h=h_x=h_y$.  We use the standard 5-stencil scheme:
\begin{equation}
	(-\Delta)_h u=-\frac{u_{i-1,j}+u_{i+1,j}-4u_{i,j}+u_{i,j-1}+u_{i,j+1}}{h^2}.
\end{equation}
We shall compute the cases $h=\,\frac{1}{4},\,\frac{1}{6},\,\frac{1}{10}\,\frac{1}{16},\,\frac{1}{25},\,\frac{1}{50}$. Denote the principal eigenvalue of $T_h$ as $\lambda_h$. For simplicity we take the exact principal eigenvalue $\lambda_*$ of $T$ on domain $\Omega$ as the one corresponding to the solution computed by MATLAB/\texttt{pdetool} with 2971838 nodes, i.e. $\lambda_*\approx 0.676147$. Denote the discrete solution of $v^{n}$ and $w^{n}$ as 
$v_h^{n}$ and $w_h^{n}$ respectively. 

\begin{rem}
	By using the \texttt{pdetool} package from MATLAB, we can compute the first eigenvalue of $-\Delta_D$ as a function of the number of nodes in the mesh. The table below collects the values of $\lambda_{(-\Delta_D)}$ and $\lambda_*=\frac{1}{\lambda_{(-\Delta_D)}}$. Here $\lambda_{(-\Delta_D)}$ denotes the eigenvalue
of $-\Delta_D$.
	\begin{table}[!h]
		\centering  %
		\begin{tabular}{|c|c|c|c|c|c|c|c|c|}  
			\hline
			\# of nodes  &168& 698 & 2941&12597&49660&218600&764685& 2971838\\ \hline
			$\lambda_{(-\Delta_D)}$&1.4888& 1.48278 &1.48015  &1.47941&1.479137&1.479022&1.478984&1.478967\\
			$\lambda_*$            &0.67164&0.67441 &0.67560  &0.67594&0.676069&0.676122&0.676139&0.676147\\
			\hline			
		\end{tabular}
		\caption{Estimate of $\lambda_{(-\Delta_D)}$ and $\lambda_*$ by Matlab/\texttt{pdetool} v.s. \# of nodes.}
		\label{lambda pde tool} 
	\end{table}
\end{rem}

\textbf{Stopping criterion. } We take the stopping criterion to be 
\begin{equation}
	\max_{i} \frac{T_h v_h^{n}}{v_h^{n}} - \min_{i} \frac{T_h v_h^{n}}{v_h^{n}}< \varepsilon,
	\label{crit 1}
\end{equation}
or equivalently, 
\begin{equation}
	\max_{i} \frac{v_h^{n}}{w_h^{n+1}} - \min_{i} \frac{v_h^{n}}{w_h^{n+1}} < \varepsilon.
	\label{crit 2}
\end{equation}
\begin{rem}
	We should mention that in Algorithm \ref{alg pde nonnegative}, the stopping criterion \eqref{crit 1} $\&$ \eqref{crit 2} are equivalent. Notice that $v^n_h=w^n_h/\Vert w^n_h\Vert$ and
	$v^{n}_h=(\lambda^{(n)}_h-T_h) w^{n+1}_h$, then
	\begin{equation}
		\lambda^{n}_h-\frac{v^{n}_h}{w^{n+1}_h}=\frac{\lambda^{n}_h w^{n+1}_h-(\lambda^{(n)}_h-T_h)w^{n+1}_h}{w^{n+1}_h}=\frac{T_h v_h^{n+1}}{v_h^{n+1}}.
	\end{equation}
    This means
    \begin{equation}
    	\max_{i} \frac{T_h v_h^{n+1}}{v_h^{n+1}} - \min_{i} \frac{T_h v_h^{n+1}}{v_h^{n+1}}=\max_{i} \frac{v_h^{n}}{w_h^{n+1}} - \min_{i} \frac{v_h^{n}}{w_h^{n+1}},
    \end{equation}
    whence the result.
\end{rem}
\vspace{1cm}

In general the eigenvalues  $\lambda_h$ exhibits a subtle dependence on the mesh size 
parameter $h$. 
By computational wisdom 
the order of $\left| \lambda_*-\lambda_h \right|$ appears to be no more than $2$, {\it i.e.} $\left| \lambda_*-\lambda_h \right|=O(h^\alpha)$ for some $0<\alpha \leqslant 2$. To extract the numerical value of $\alpha$ we employ the formula
\begin{equation}
	\label{3.23a}
	\alpha \approx \frac{\log(\left| \lambda_*-\lambda_{h_1} \right|)-\log(\left| \lambda_*-\lambda_{h_2} \right|)}{\log(h_1)-\log(h_2)}.
\end{equation}
(See Table \ref{lambda error}.) 

The exact $\lambda_h$ is computed by using power method ($v_h^{n+1}=T_h v_h^{n}$) until $\varepsilon<1e-14$ in $(\ref{crit 1})$, {\it i.e.}
\begin{equation}
 \max_{i} \frac{T_h v_h^{n}}{v_h^{n}} - \min_{i} \frac{T_h v_h^{n}}{v_h^{n}}< 10^{-14}.
\end{equation}
According to (\ref{1.5a}) we know that
\begin{equation}
 \min_{i} \frac{T_h v_h^{n}}{v_h^{n}}<\lambda_h<\max_{i} \frac{T_h v_h^{n}}{v_h^{n}}.
\end{equation}
Take $\lambda_h\approx \max_{i} \frac{T_h v_h^{n}}{v_h^{n}}$, 
we obtain the  value of $\lambda_h$ with $10^{-14}$ precision.
\begin{table}[htb]
	\centering  %
	\begin{tabular}{|c|c|c|c|}  
		\hline
		h &$\lambda_h$& Error & Order \\ \hline
		1/4&0.670972 & 5.174e-03 & --  \\
		1/6&0.672725 & 3.422e-03 &1.02\\
		1/10&0.674231& 1.916e-03 &1.13\\
		1/16&0.675061& 1.086e-03 &1.20\\
		1/25&0.675525& 0.6213e-03&1.25\\
		1/50&0.675893& 0.2537e-03&1.29\\
		\hline			
	\end{tabular}
	\caption{The error $\left| \lambda_*-\lambda_{h} \right|$ and its order}
	\label{lambda error} 
\end{table}

\textbf{Iteration steps.} We take initial vector $v^0=T1$, {\it i.e.}  solving 
$-\Delta v^0=1$ with the Dirichlet boundary condition. We shall use two different tolerance error thresholds $\varepsilon=1e-14$ or $\varepsilon=\,h^2/10$. In Table \ref{tab step}, the second and third row correspond to the case $\varepsilon=1e-14$, and the fourth and fifth correspond to \footnote{In view of \eqref{3.23a}, it is natural to take the error tolerance threshold to be $\frac{1}{10}h^2$.} $\varepsilon=h^2/10$. The error is  computed as $|\lambda_h-\lambda_h^{(n)}|$. In general it is observed that the convergence takes place within $2$ to  $4$ steps. An interesting observation is that the number of iterations remains  almost independent of the mesh size parameter $h$.

\begin{table}[htb]
	\centering  %
	\begin{tabular}{|c|c|c|c|c|c|c|c|}  
		\hline
		&h&$1/4$&$1/6$&$1/10$&$1/16$&$1/25$&$1/50$\\ \hline
		Criterion (\ref{crit 2}) &Iterations& 4     & 4   &4& 4   &4&4 \\ \cline{2-8}
		$\varepsilon=$1e-14&Error& 2.22e-15 &9.99e-16&1.84e-14&9.38e-14&7.62e-14&7.63e-13\\ \hline
		Criterion (\ref{crit 2}) &Iterations& 2   &2  &2& 2  &3&3 \\ \cline{2-8}
		$\varepsilon=h^2/10$&Error& 6.57e-05 &6.64e-05&6.65e-05&6.63e-05&8.48e-09&8.46e-09\\ \hline
	\end{tabular}
	\caption{Number of iterations and errors of the proposed method.}
	\label{tab step}
\end{table}

\textbf{Quadratic convergence.} We take $h=\frac{1}{6},\,\frac{1}{16},\,\frac{1}{50}$ to test the convergence rate. The stopping criterion is (\ref{crit 2}) with tolerance threshold$\varepsilon=1e-14$. It is observed that $\lambda_h^{(n)}$ converges to $\lambda_h$ quadraticaly. 
We employ the empirical formula
\begin{equation}
	\frac{\log(|\lambda_h^{(n+1)}-\lambda_h|)-\log(|\lambda_h^{(n)}-\lambda_h|)}{\log(|\lambda_h^{(n)}-\lambda_h|)-\log(|\lambda_h^{(n-1)}-\lambda_h|)}
	\label{test order}
\end{equation}
to compute the order of convergence.
\begin{table}[htb]
	\centering  %
	\begin{tabular}{|c|c|c|c|c|c|c|}  
		\hline
		\hline
		& \multicolumn{2}{|c|}{$h=1/6$} &\multicolumn{2}{|c|}{$h=1/16$}&\multicolumn{2}{|c|}{$h=1/50$}\\ \cline{2-7}
		&error&order&error&order&error&order\\ \hline
		1& 5.903e-03     & --   & 5.921e-03  & --   &  5.927e-03  &--  \\
		2& 6.639e-05     & --   &6.633e-05& --   &6.635e-05&-- \\
		3& 8.557e-09 &1.996&8.481e-09&1.996&8.467e-09&1.996\\
		4& 9.992e-16 &1.782&9.381e-14&1.273&7.627e-13&1.039\\
		\hline			
	\end{tabular}
	\caption{Order of convergence of the first four steps. The fourth step has reached the machine error, and the corresponding result is an outlier.}
	\label{tab order}
\end{table}

\textbf{Comparison with Rayleigh quotient.} As is well known, the standard Rayleigh quotient iteration has spectacular cubic convergence, if the initial condition is  sufficient close to the target eigen-pair. 

For comparison we take two different initial conditions $v^0=1$ or $v^0=T1$ to test the cubic convergence. Since the Rayleigh quotient does not preserve the positivity of $v^n$ in general, we shall enforce the stopping criterion 
\begin{align}
	\label{3.25a}
	|\lambda^{(n)}-\lambda_*|<\varepsilon.
\end{align} 
When we implement \eqref{3.25a} for the discretized Dirichlet Laplacian, the corresponding stopping criterion takes the form
\begin{align}
	\label{3.26a}
	|\lambda^{(n)}_h-\lambda_h|<\varepsilon.
\end{align} 

\begin{algorithm}[h]
	\caption{Rayleigh quotient iteration}
	\label{alg2 rayl nonnegative}
	\begin{algorithmic}
		
		\STATE {
			\begin{enumerate}[\textbf{Step} 1]
				\item
				Choose the initial vector $v^{0}>0$ and compute
				\begin{equation}\label{alg2 rl z0}
					\lambda^{(0)}= \frac{(v^0,\,Tv^0)}{(v^0,\,v^0)}
				\end{equation}
				
				\item
				{For $n=1,2,\dots$}, solve the linear equation
				\begin{equation}\label{alg2 rl system}
					\Big(\lambda^{(n-1)}I-T \Big)\omega^{n}=v^{n-1}
				\end{equation}
				and calcuate 
				\begin{align}
					v^{n}=\frac {\omega^{n}}{\|\omega^{n}\|}.
				\end{align}
				
				\
				Update $\lambda^{(n)}$ with
				\begin{equation}\label{alg2 rl zn}
					\lambda^{(n)}= (v^n,\,Tv^n)
				\end{equation}
				
				\item
				Enforce the stopping criterion and output 
				$ (\lambda^{(n)}, \; v^{n})$.
		\end{enumerate}}
	\end{algorithmic}
\end{algorithm}

\begin{table}[htb]
	\centering  %
	\begin{tabular}{|c|c|c|c|c|c|}  
		\hline
		\multicolumn{2}{|c|}{ }&\multicolumn{2}{|c|}{$v^0=1$}&\multicolumn{2}{|c|}{$v^0=T1$}\\ \hline
		\multicolumn{2}{|c|}{h}&$1/16$&$1/50$&$1/16$&$1/50$\\ \hline
		Rayleigh&Iterations& 4     & 4   &2& 2   \\ \cline{2-6}
		quotient&Error& 6.66e-16 &1.33e-15&5.55e-16&1.22e-15\\ \hline
		Our &Iterations& 5   &5  &4& 4  \\ \cline{2-6}
		algorithm&Error& 3.11e-14 &5.96e-13&2.42e-14&3.30e-13\\ \hline
	\end{tabular}
	\caption{The comparison of Rayleigh quotient and our method. The error is $|\lambda_h-\lambda_h^{(n)}|$. The stopping criterion is $|\lambda_h^{(n)}-\lambda_h|<1e-12$.}
	\label{tab ral step}
\end{table}

\begin{table}[htb]
	\centering  %
	\begin{tabular}{|c|c|c|}  
		\hline
		\hline
		& \multicolumn{2}{|c|}{$h=1/50$} \\ \cline{2-3}
		&Error&Order\\ \hline
		1& 8.499e-02     & --     \\
		2& 2.9410e-03     & --   \\
		3& 8.025e-08 &3.124\\
		4& 1.332e-15 &1.704\\
		\hline			
	\end{tabular}
	\caption{Convergence order of Rayleigh quotient iteration. Here the initial condition is $v^0=1$.}
	\label{tab ral order}
\end{table}

From Table \ref{tab ral order}, it is observed that the iterates appear to converge to the principal eigenvalue with cubic convergence.

\end{document}